\documentclass[11pt,a4paper]{article} 

\newtheorem{theorem}{Theorem}

\newtheorem{lemma}[theorem]{Lemma}

\newtheorem{question}[theorem]{Question}
\newenvironment{proof}{\noindent{\bf Proof.}\ }{$\bullet$\medskip\par}
\usepackage{amssymb}

\title{\bf Non vanishing of theta functions
and sets of small multiplicative energy}
\author{Marc Munsch \\
5010 Institut f\"{u}r Analysis und Zahlentheorie \\
8010 Graz, Steyrergasse 30, Graz, Austria \\
munsch@math.tugraz.at}

\date{\today}

\begin{document}
\bibliographystyle{alpha}
\maketitle


\begin{abstract}
Let $\chi$ range over the $(p-1)/2$ even Dirichlet characters modulo a prime $p$ and denote by $\theta (x,\chi)$ the associated theta series. The asymptotic behaviour of the second and fourth moments proved by Louboutin and the author implies that there exists at least $ \gg p/ \log p$ characters such that the associated theta function does not vanish at a fixed point. Constructing a suitable mollifier, we improve this result and show that there exists at least $ \gg p/ \sqrt{\log p}$ characters such that $\theta(x,\chi) \neq 0$ for any $x>0$.
We give similar results for odd Dirichlet characters mod $p$.
\end{abstract}

\bibliographystyle{plain}
\maketitle

\section{Introduction}
\hspace{\parindent} The distribution of values of $L$- functions is a deep question in number theory which has various important repercussions for the related attached arithmetic, algebraic and geometric objects. It is certainly of great importance in number theory. The main reason comes from the fact that these values and particularly the central ones hold a lot of fundamental arithmetical information, as illustrated for example by the famous Birch and Swinnerton-Dyer Conjecture conjecture (\cite{BSD1,BSD2}). It is widely believed that they should not vanish unless there is an underlying arithmetic reason forcing it and this should occur very rarely when considered inside suitable families. Consider the Dirichlet $L$- functions associated to Dirichlet characters $$L(s,\chi):=\sum_{n \geq 1} \frac{\chi(n)}{n^s} \hspace{5mm} \Re(s) > 1.$$ In this case, there exists no algebraic reason forcing the $L$- function to vanish at $s=1/2$. Therefore it is certainly expected that $L(1/2) \neq 0$ as firstly conjectured by Chowla \cite{Chowla} for quadratic characters. In the last century, the notion of family of $L$- functions has been important  both as a heuristic guide to understand or guess many important statistical properties of $L$-functions. One of the main analytic tools is the study of moments and various authors have obtained results on the mean value of these $L$-series at their central point $s=1/2$. The asymptotic of the first two moments (as well as the fourth moment) is known: 
$$\sum_{\chi\bmod p}
\vert L(1/2,\chi )\vert^{2}
\sim p\log p$$
(see \cite[Remark 3]{Rama}, 
\cite[Th. 3]{Bal}, 
or \cite{HB81a} for a more precise asymptotic expansion). This imply directly that there is a reasonable proportion of characters such that the $L$- function does not vanish at the central point. \\

  Using the method of mollifiers, it was first proved by Balasubramanian and Murty \cite{BalaMurty} that there exists a positive proportion of characters such that the $L$- function does not vanish at $s=1/2$. Their result was improved and greatly simplified by Iwaniec and Sarnak \cite{ISoriginal} enabling them to derive similar results for families of automorphic $L$-functions \cite{ISauto}. Since then, a lot of technical improvements and generalizations have been carried out, see for instance \cite{Bui,Khan,Soundquad}.\\

As initiated in previous works \cite{LM,Debrecen,thetalow,thetaupp}, we would like to obtain similar results for moments of values at $x=1$ 
of the theta functions $\theta (x,\chi )$ associated with such Dirichlet $L$-functions. It was conjectured in \cite{efficient} that $\theta(1,\chi) \neq 0$ for every primitive character (see \cite{Zagier} for a case of vanishing in the non-primitive case). Using the computation of the first two moments of these theta functions at the central point $x=1$, Louboutin and the author \cite{LM} obtained that $\theta (1,\chi)\neq 0$ for at least $p/\log p$ even characters modulo $p$ (for odd characters, it was already proven by Louboutin in \cite{CRAS}). As in the case of $L$- functions, we would like to obtain a positive proportion of such characters. However, one backdraw in this situation is that the theta function does not have a representation as an Euler product which suggested the construction of mollifiers for Dirichlet $L$- functions. Thus, we need to proceed somehow differently in order to construct the mollifiers. 
Our goal in this note is to provide an argument which does not produce a positive proportion but improve the result coming from the evaluation of the second and fourth moments. Precisely, we prove 
\begin{theorem}\label{mainth}
Let $x>0$. For all sufficiently large prime $p$, there exists at least $\gg p/ \sqrt{\log p}$ even characters $\chi$ such that $\theta(x,\chi) \neq 0$.
\end{theorem} Theorem \ref{mainth} follows from a classical application of the Cauchy-Schwarz inequality: \begin{equation}\label{Cauchy}
M_1(p)^2 \leq \left(\sum_{\chi \bmod p, \chi(-1)=1 \atop \theta(x,\chi) \neq 0} 1 \right) M_2(p) 
 \end{equation} where 
we consider some mollified moments
$$
M_1(p) :=\sum_{\chi\bmod p \atop \chi(-1)=1}M(\chi)\theta (x,\chi) \textrm{   and   } M_2(p):= \sum_{\chi\bmod p \atop \chi(-1)=1}\vert M(\chi)\theta (x,\chi)\vert^2. $$ The mollifier $M(\chi)$ will be chosen as a suitable Dirichlet polynomial builded on a multiplicative subset of integers. The main idea is of combinatorial nature and relies on minimizing some GCD sums or equivalently the multiplicative energy of some well-chosen set. \\

In Section \ref{Def}, we review some previous work as well as introduce the mollifiers. In Section \ref{GCD} and \ref{goodset}, we discuss the reduction of the problem to GCD sums which help us to prove Theorem \ref{mainth} in Section \ref{ProofTh}. Section \ref{odd} concerns the case of odd characters and in Section \ref{combiproblems} we address some open questions regarding the combinatorial problem which may lead to improve our main theorem.

\section{Definitions and previous results}\label{Def}
Let us first restrict ourselves to the case of even Dirichlet characters, we refer the reader to Section \ref{odd} for the case of odd Dirichlet characters.
Let $X_p^+$ be the subgroup of order $(p-1)/2$ of the even Dirichlet characters mod $p$.
We set 
$$\theta (x,\chi)
=\sum_{n\geq 1} \chi (n)e^{-\pi n^2x/p}, 
\ \ \ \chi\in X_p^+.$$ 
For $\chi\neq 1$ we have
$$\theta (x,\chi ) 
={W_\chi\over\sqrt x}\theta (1/x,\bar\chi )
\ \ \ \ (x>0)$$ 
for some explicit complex number $W_\chi$ of absolute value equal to one.
(e.g. see \cite[Chapter 9]{Dav}). 
In particular, if $\theta (1,\chi)\neq 0$, 
then we can efficiently compute numerical approximations to 
$W_\chi =\theta (1,\chi)/\overline{\theta (1,\chi)}$ 
(see \cite{efficient} for an application). 
In order to prove that $\theta (1,\chi)\neq 0$ for many of the $\chi\in X_p^+$, one may proceed as usual and study the behavior of the moments of these theta functions at the central point $x=1$ 
of their functional equations:
$$S_{2k}^+(p) :=\sum_{\chi\in X_p^+}\vert\theta (1,\chi)\vert^{2k}, \hspace{10mm} k > 0.$$ 

Using the computation of the second and fourth moment, it was proved in \cite{LM} that $\theta (1,\chi)\neq 0$ for at least $\gg p/\log p$ of the $\chi\in X_p^+$. Lower bounds of good expected order for the moments were obtained in \cite{thetalow} as well as nearly optimal upper bounds conditionally on GRH in \cite{thetaupp}. This can be related to recent results of~\cite{HarperMaks}, where the authors obtain the asymptotic behaviour of moments of Steinhaus random multiplicative function (a multiplicative random variable whose values at prime integers are uniformly distributed on the complex unit circle). This can reasonably be viewed as a random model for $\theta(1,\chi)$. Indeed, the rapidly decaying factor $e^{-\pi n^2/q}$ is mostly equivalent to restrict the sum over integers $n \le n_0(q)$ for some $n_0(q) \approx \sqrt{q}$ and the averaging behavior of $\chi(n)$ with $n\ll q^{1/2}$ is essentially similar to that of a Steinhaus random multiplicative function. As noticed by Harper, Nikeghbali and Radziwill in \cite{HarperMaks}, an asymptotic formula for the first absolute moment $\displaystyle{\sum_{\chi \in X_p^+} \vert \theta(1,\chi)\vert}$ would probably implies the existence of a positive proportion of characters such that $\theta(1,\chi) \neq 0$. Though, quite surprisingly,  Harper proved recently both in the random and deterministic case that the first moment exhibits unexpectedly more than square-root cancellation \cite{Harpelow,Harperhigh}
$$\frac{1}{p-1} \sum_{\chi} \left\vert \sum_{n\leq N} \chi(n)\right\vert \ll \frac{\sqrt{N}}{\min\left\{\log \log N, (\log \log \log q)^{1/4}\right\}}.$$ This shows that this approach would not in any case provide the existence of a positive proportion of ``good" characters. In this note, we adapt another approach in order to improve on our previous results. 

\vspace{2mm} 
Let $M$ be a parameter which will be fixed later. For any even character $\chi$, let us define 
 \begin{equation}\label{weight} M(\chi)=\sum_{m\leq M}c_m \bar\chi(m)\end{equation} where $c_m$ denotes the indicator function of some multiplicative set of integers $\mathcal{A}$, meaning that $m,n \in \mathcal{A}$ implies $mn\in \mathcal{A}$. We consider the first and second mollified moments 
\begin{equation}\label{mollif}
M_1(p) :=\sum_{\chi\in X_p^+}M(\chi)\theta (x,\chi) \textrm{   and   } M_2(p):= \sum_{\chi\in X_p^+}\vert M(\chi)\theta (x,\chi)\vert^2.
\end{equation} 
While it seems plausible that we could obtain in some cases precise asymptotical formulas, we only give bounds in order to simplify the presentation.  The main technical result is to suitable choose a set ${\cal A}$
giving simultaneously a good lower bound for $M_1(p)$ and a good upper bound for $M_2(p)$.
Precisely, for a suitable choice of $\mathcal{A}$, we have asymptotically
\begin{equation}\label{boundsmolli}
M_1(p) \gg \frac{p^{3/2}}{\sqrt{\log p}} \textrm{   and   } M_2(p) \ll \frac{p^2}{\sqrt{\log p}}.
\end{equation} As already noticed, Theorem \ref{mainth} follows from Cauchy-Schwarz inequality (\ref{Cauchy}) combined with (\ref{boundsmolli}). In the next sections, we will explain how to construct a good set which verifies (\ref{boundsmolli}) and address a related combinatorial problem which may lead to an improvement of our result.

\section{Reduction of the problem to GCD sums}\label{GCD}
\subsection{Lower bound on the first mollified moment}

Let us recall the classical orthogonality relations for the subgroup of Dirichlet even characters $X_p^+$:
$$\sum_{\chi\in X_p^+}\chi (m)\bar\chi (n)
=\cases{
(p-1)/2&if $m\equiv\pm n\bmod p$ and $\gcd (m,p)=1$,\cr
0&otherwise.\cr}$$
Due to the fast decay of the exponential term in $\theta(x,\chi)$, the main contribution to $M_1(p)$ comes from the terms less than $\sqrt{p}$, leading us to choose $M=\sqrt{p}$. 
It follows that 

\begin{eqnarray}\label{lowM1}
M_1(p) &=& \sum_{\chi\in X_p^+} \sum_{m\leq \sqrt{p}} \bar\chi (m)c_m \sum_{n\geq 1}\chi(n)e^{-\pi n^2x/p} \nonumber \\
&=& \frac{p-1}{2}\sum_{m\leq \sqrt{p}}c_m e^{-\pi m^2x/p} \nonumber \\
& \gg & p \sum_{m\leq \sqrt{p}} c_m 
\end{eqnarray} where we used the fact that $e^{-\pi m^2x/p} \gg 1$ for integers less than $\sqrt{p}$. 

The problem boils down to choose a subset of integers $\mathcal{A}$ of sufficiently high density in order to maximize $M_1(p)$ with the condition that it minimizes the second mollifier. We remark from the above inequality (\ref{lowM1}) that 
\begin{equation}\label{lowfirst} M_1(p) \gg p\vert\mathcal{A}\cap \left[1,M\right] \vert \end{equation} where $\vert\mathcal{A}\cap \left[1,M\right] \vert$ denotes the number of elements less than $M$ in $\mathcal{A}$. By an abuse of notation, we will use $\vert \mathcal{A}\vert$ in the following.

\subsection{Upper bound for the second mollified moment and multiplicative energy}

The evaluation of the second moment is a bit more intricate.

\begin{eqnarray}\label{upperM2}
M_2(p) &= &  \sum_{\chi\in X_p^+}\sum_{m_1,m_2 \leq \sqrt{p}} \bar\chi (m_1)\chi(m_2)c_{m_1}c_{m_2} \sum_{n_1,n_2 \geq 1}\chi(n_1^{-1}n_2)e^{-\frac{\pi (n_1^2+n_2^2)x}{p}} \nonumber \\
&=& \frac{p-1}{2} \sum_{m_1,m_2 \leq \sqrt{p}} c_{m_1}c_{m_2}\sum_{m_1 n_1 \equiv\pm m_2 n_2\pmod p \atop (n_1 n_2, p)=1}e^{-\pi (n_1^2+n_2^2)x/p} \nonumber\\
& \ll &  \sum_{m_1,m_2 \leq \sqrt{p}} c_{m_1}c_{m_2}\sum_{n\geq 1} E_{m_1,m_2}(n)e^{-\pi n x/p}
\end{eqnarray} where we used partial summation and 

$$E_{m_1,m_2}(n)= \sum_{n_1^2+n_2^2\leq n\atop{m_1n_1 = \pm m_2n_2 \bmod p}} 1.$$

In \cite{LM}, in order to compute an asymptotic formula for the fourth moment of theta functions, the authors showed that the main contribution comes from the solutions $m_1n_1 = m_2n_2$ and obtained a precise asymptotic formula for the related counting function 
$$\vert\{m_1n_1=m_2n_2, m_1^2+n_1^2+m_2^2+n_2^2 \leq x\}\vert \sim {3\over 8}x\log x.$$  
If we want to improve on this result, we have to minimize the effect of this logarithmic term. Let us recall a related bound which includes all the solutions modulo $p$. Taking initial intervals in \cite[Lemma $1$]{Ayyad}, we have 
\begin{lemma}\label{energy}
Suppose $m_1\leq m_2$ and $x\leq p$, then we have the following bound
$$ \vert \left\{ n_1,n_2 \leq x, m_1n_1 = \pm m_2n_2 \bmod p\right\}\vert \ll \left(1+\frac{x(m_1,m_2)}{m_2}\right)\left(1+\frac{x m_2}{(m_1,m_2)p}\right)   $$
where as usual $(m_1,m_2)$ denotes the greatest common divisor of $m_1$ and $m_2$. 
\end{lemma}

We deduce immediately the bound

\begin{equation}\label{energybound} E_{m_1,m_2}(n) \ll 1+\frac{n}{p}+\frac{\sqrt{n} m_2}{p(m_1,m_2)}+\sqrt{n}\frac{(m_1,m_2)}{m_2}.
\end{equation}

Truncating the series in (\ref{upperM2}) up to $p^2$, inserting the bound (\ref{energybound}) in (\ref{upperM2}) and using a comparison with an integral, the first three terms in the right hand side of (\ref{energybound}) contribute at most $p\left(\sum_{m\leq \sqrt{p}}c_m\right)^2$.  The last term is more problematic and we can summarize this in the inequality

\begin{equation}\label{M2final}
M_2(p) \ll p\left(\sum_{m\leq \sqrt{p}}c_m\right)^2+ p^{3/2}\sum_{m_1\leq m_2 \leq \sqrt{p}} c_{m_1}c_{m_2}\frac{(m_1,m_2)}{m_2}.
\end{equation}

We recall that we assumed that the weights are all equal to $1$ or $0$ and are totally multiplicative meaning that our mollifier is contructed as the counting function of some multiplicative set $\mathcal{A}$. The first sum on the right hand side of the inequality (\ref{M2final}) gives the harmless contribution to $M_2(p)$
$$ p\left(\sum_{m\leq \sqrt{p}}c_m\right)^2 = p\vert \mathcal{A}\vert ^2.$$ Hence, using Cauchy-Schwarz inequality with (\ref{lowfirst}) and (\ref{M2final}), we deduce the following lower bound
\begin{equation}\label{propbound}
\sum_{\chi \in X_p^+ \atop \theta(x,\chi) \neq 0} 1 \gg  \left(p^2\vert \mathcal{A}\vert ^2\right)/\left(p^{3/2}\sum_{m_1 \leq m_2 \leq \sqrt{p}} c_{m_1}c_{m_2}\frac{(m_1,m_2)}{m_2}\right).
\end{equation}
We need to construct a set $\mathcal{A}$ in order to maximize this ratio.

\section{Good sets minimizing GCD sum}\label{goodset}

Let us consider the general setting of a subset $\mathcal{B}\subset \left[1,N\right]$ of integers. We are interested to minimize the quantity 
$$S(\mathcal{B}):=\sum_{m_1,m_2 \in \mathcal{B} \atop m_1 \leq m_2}\frac{(m_1,m_2)}{m_2}.$$
More precisely, in view of (\ref{propbound}), we want to maximize in terms of $N$ the quantity 

$$R(\mathcal{B}):= N\frac{\vert\mathcal{B}\vert^2}{S(\mathcal{B})}.$$  Restricting the sum to the couples $(m_1,m_2)$ where $m_1 \mid m_2$ we have 
$$S(\mathcal{B}) \geq \sum_{m_1,m_2 \in \mathcal{B}\atop m_1 \mid m_2} \frac{m_1}{m_2} \geq \sum_{m_2 \in \mathcal{B}} \sum_{d \mid m_2}\frac{1}{d} =\sum_{m\in \mathcal{B}} \sigma_{-1}(m) \geq \vert\mathcal{B}\vert.$$

In particular, we have the trivial bound
\begin{equation}\label{gcdupp} R(\mathcal{B}) \leq N\vert \mathcal{B}\vert. \end{equation} 

Assume that $\mathcal{B}$ is a set such that $m_2\in\mathcal{B}$ and $m_1 \mid m_2$ implies $m_1\in\mathcal{B}$, then the divisor sum over $d$ will be complete.  Therefore, we need to construct a set such that, on average, every element has few small divisors. Precisely, we are looking for a set of reasonable density such that the GCD sum $S(\mathcal{B})$ is not too large. 

Let us discuss two extremal cases. Assume  $\mathcal{B}$  to be the set consisting of all the primes less than $N$, we have $S(\mathcal{B}) \ll N/ \log N$. The main contribution comes from the diagonal terms $p=q$, in the other case $p$ and $q$ are coprimes and the sum is small. Thus, by the prime number theorem, $R(\mathcal{B}) \approx N^2 / \log N$. Even though we are able to have good control on the GCD sum, the set of primes has a too small density which prevents us to save any logarithm in $R(\mathcal{B})$. Another extreme case is to take all the integers up to $N$, a short computation shows that the GCD sum $S(\mathcal{B})\approx N\log N$ and again $R(\mathcal{B}) \approx  N^2 / \log N$. This is mainly equivalent as considering the fourth moment of theta functions like in $\cite{LM}$. We thus seek for an intermediate case of an high density set and relatively well controlled GCD sum $S(\mathcal{B})$. 

\subsection{Integers free of small prime factors}

 As we remarked, taking the prime numbers is a good choice in order to have an optimal small GCD sum. Moreover, we can view the primes $\leq N$ as the set obtained after sieving out the first $N^{1/2}$ primes. Our idea is to increase the size of this set sieving out small primes but at the same time control the GCD sum. Practically we will consider the set of numbers free of prime factors smaller than $y$ for some parameter $y$ which will be fixed later. Denote by $P^{-}(n)$ the smalllest prime factor of an integer $n$. Let us write 

$$ \Phi(x,y):=\left\vert \left\{n \leq x : P^{-}(n)> y\right\}\right\vert \hspace{7mm} (x\geq y \geq 2).$$ As for instance proved in \cite[Chapter $15$]{Tencourse}, Brun's sieve implies the asymptotic formula, valid uniformly for $x\geq y \geq 2$,

\begin{equation}\label{asymfree} \Phi(x,y) = \frac{x}{\zeta(1,y)}\left\{1+O\left(\frac{1}{(\log y)^2}\right)\right\}  \hspace{5mm}  (2 \leq y \leq x^{1/10 \log_2 x})\end{equation} where $$\zeta(1,y)=\prod_{p\leq y}(1-1/p)^{-1}$$ is the partial zeta function.
The following simple observation enlights the fact that if we want to construct a set such that two distinct integers have bounded gcd's, we can remove the small prime factors.
\begin{lemma}\label{gcdbound} 
A pair of integers $m \leq n \leq x$ such that $P^{-}(mn) > y$  verifies either $m\mid n$ or $(m,n) < x/y$. 
\end{lemma}
\begin{proof} Suppose $m \nmid n$, then $(m,n)$ is a proper divisor of $m$. Thus it has to divide $m/p$ for some prime $p >y$ which conludes the proof. \end{proof}

\subsection{Sieve results} \label{sec:sieve}

Let $\mathcal{P}=\left\{p \textrm{ prime such that }p\leq  y\right\}$. Then we have the classical following result.

\begin{lemma}\label{sieve1}
 We define $\mathcal{A}=\left\{1\leq n \leq N: (n,p)=1 \textrm{ for all p } \in \mathcal{P} \right\}$, where $N$ is an arbitrary integer such that $y\leq N.$ There exists an absolute constant $c$ such that  

\begin{equation}\label{upperprodbis} \#A \leq c N \prod_{p\in \mathcal{P}} \left(1-\frac{1}{p}\right).\end{equation}
\end{lemma}

\begin{proof} This is an application of Brun's sieve which follows from \cite[Theorem $2.2$, Equation ($5.2$)]{halbrich} or can also be deduced from Selberg's sieve \cite[Theorem $3.5$]{halbrich}. \end{proof}

Sieving with logarithmic weights is very elementary.
\begin{lemma}\label{logsieve}
Assume as before that $N$ is an arbitrary integer such that $y\leq N.$ Then we have
$$\sum_{n \in \mathcal{A}}\frac{1}{n} \ll \log N \prod_{p\in \mathcal{P}} \left(1-\frac{1}{p}\right).$$
\end{lemma}
\begin{proof}
This is immediate 
\begin{eqnarray*}
 \sum_{n\leq N \atop p\mid n \Rightarrow p>y} \frac{1}{n} &\leq &\prod_{y  < p \leq N} \left(1-\frac{1}{p}\right)^{-1} = \prod_{p\leq N} \left(1-\frac{1}{p}\right)^{-1} \prod_{p\leq y} \left(1-\frac{1}{p}\right) \\
 & \sim & e^{\gamma}\log N \prod_{p \leq y} \left(1-\frac{1}{p}\right)\end{eqnarray*} by Mertens' theorem.
\end{proof}

\subsection{Estimate of $R(\mathcal{B})$ for the set of integers free of small prime factors}

We want to give an upper bound for $R(\mathcal{B})$  in this intermediate setting where $\mathcal{B}$ is the set of integers free of small prime factors less than some parameter $y$. By Lemma \ref{gcdbound}, we know that either $m_1 \mid m_2$ or $(m_1,m_2) < N/y$. 

\subsubsection{Bound in the case $m_1\mid m_2$}
We have
$$\sum_{m_1,m_2 \leq N \atop m_1 \mid m_2} \frac{m_1}{m_2}=\sum_{m_2 \leq N} \sum_{d \mid m_2}\frac{1}{d}.$$ 

We separate the contribution from $d=1$ to the one coming from proper divisors of $m_2$. We notice that proper divisors of $m_2$ are greater than $y$. Thus,

\begin{eqnarray*}
\sum_{m_2 \leq N \atop P^{-}(m_2) > y} \sum_{d \mid m_2}\frac{1}{d} &=& \sum_{m_2 \leq N \atop P^{-}(m_2) > y} 1 + \sum_{m_2 \leq N \atop P^{-}(m_2) > y} \sum_{d \mid m_2 \atop d \neq 1}\frac{1}{d} \\
& \ll & \vert \mathcal{B}\vert + \sum_{y \leq d \leq N} \frac{1}{d} \sum_{k \leq N/d} 1 \\ 
& \ll & \vert \mathcal{B}\vert  + \log N + N/y.
\end{eqnarray*}
For $y$ in an intermediate range, the main contribution comes from $\vert \mathcal{B}\vert $. This gives the first inequality $S(\mathcal{B}) \ll \vert \mathcal{B}\vert$ which is optimal by (\ref{gcdupp}). 

\subsubsection{Contribution from proper divisors}
By Lemma \ref{gcdbound}, we have 
$$ \sum_{m_1,m_2 \in \mathcal{B}}\frac{(m_1,m_2)}{m_2}=\sum_{\delta \leq N/y \atop \delta \in\mathcal{B}} \delta \sum_{m_1,m_2 \in \mathcal{B} \atop (m_1,m_2)= \delta} \frac{1}{m_2} = \sum_{\delta \leq N/y \atop \delta \in\mathcal{B}}\sum_{m_1,m_2 \leq N/\delta \atop m_1,m_2 \in \mathcal{B} ;(m_1,m_2)= 1 } \frac{1}{m_2}.$$

We can forget the coprimality condition, so we have two independent sums over sifted sets. We seek for an upper bound sieve for the double sum over $m_1$ and $m_2$. Noticing that $m_2 \geq y$ if $m\in \mathcal{B}$ and $N/\delta \geq y$, we have using twice Lemma \ref{sieve1} (first for the sum over $m_1$ and then for the sum over $m_2$)

\begin{eqnarray*}
 \sum_{\delta \leq N/y \atop  \delta \in\mathcal{B}}\sum_{m_1\leq m_2 \leq N/\delta \atop m_1,m_2 \in \mathcal{B} ; (m_1,m_2)= 1} \frac{1}{m_2} & \ll & \sum_{\delta \leq N/y \atop \delta \in\mathcal{B} }\sum_{m_2\leq N/\delta \atop m_2 \in \mathcal{B}} \prod_{p\leq y} \left(1-1/p\right) \\
 &\ll & \sum_{\delta \leq N/y \atop \delta \in\mathcal{B}} N/\delta \prod_{p\leq y} \left(1-1/p\right)^2. \end{eqnarray*} We conclude using Lemma \ref{logsieve} to handle the sum over $\delta$ (assuming $N/y \geq y$)

$$ \sum_{\delta \leq N/y \atop \delta \in\mathcal{B}} N/\delta \prod_{p\leq y} \left(1-1/p\right)^2  \ll  N\log N \prod_{p\leq y} \left(1-1/p\right)^3.$$
 
Combining all the previous inequalities, and assuming that the parameter $y$ verifies $\exp({\log^{\epsilon}N})\leq y \leq N^{\epsilon}$, we have

$$S(\mathcal{B}) \ll \max\left\{\vert \mathcal{B}\vert,N\log N \left(\frac{\vert\mathcal{B}\vert}{N}\right)^3\right\}$$ where we used the asymptotic from (\ref{asymfree})

$$ \vert \mathcal{B}\vert \sim N\prod_{p\leq y}\left(1-1/p\right).$$This implies

$$R(\mathcal{B}) \gg \min\left\{N\vert\mathcal{B}\vert,\frac{N^3}{\vert \mathcal{B}\vert \log N}\right\}.$$We can rewrite this in a more pleasant way. Setting $\vert \mathcal{B}\vert = \alpha N$, we have

\begin{equation}\label{finalbound} R(\mathcal{B}) \gg \min\left\{ N^2 \alpha, \frac{N^2}{\alpha \log N}\right\}\end{equation} leading to the optimal choice $\alpha = \frac{1}{\sqrt{\log N}}$.  

\section{Proof of Theorem \ref{mainth}}\label{ProofTh}

In our context, we have $N=\sqrt{p}$ and we can choose $\mathcal{A}$ in (\ref{weight}) as the set of numbers having all their prime factors greater than $\exp({\sqrt{\log p}})$. Using the asymptotical formula (\ref{asymfree}) and Mertens' theorem, we have $\vert\mathcal{A}\vert \sim \frac{\sqrt{p}}{\sqrt{\log p}}$.  Thus inserting in (\ref{finalbound}),
$$ R(\mathcal{A}) \gg N^2 \alpha = \frac{p}{\sqrt{\log p}}.$$ This leads to the following lower bound on the proportion of non vanishing  
 $$ M_1(p)^2/M_2(p) \gg \frac{p}{\sqrt{\log p}}, $$ which concludes the proof by (\ref{Cauchy}).

\section{The case of odd characters}\label{odd}

If $\chi$ mod $p$ is odd, 
then we set
$$\theta (x,\chi)
=\sum_{n\geq 1} n\chi (n)e^{-\pi n^2x/p}$$
and we have the relations
$$\sum_{\chi\in X_p^-}\chi (a)\bar\chi (b)
=\cases{
(p-1)/2&if $b\equiv a\pmod p$ and $\gcd (a,p)=1$,\cr
-(p-1)/2&if $b\equiv -a\pmod p$ and $\gcd (a,p)=1$,\cr
0&otherwise,\cr}$$
where $X_p^-$ is the set of the $(p-1)/2$ odd Dirichlet characters mod $p\geq 3$. 

Using a similar method and partial summation, we can show that $\theta (x,\chi)\neq 0$ 
for at least $\gg p/\sqrt{\log p}$ of the $(p-1)/2$ characters $\chi\in X_p^-$.
 
\section{Combinatorial open questions and consequences}\label{combiproblems}
 
To summarize, we constructed a set $\mathcal{B}\subset [1,N]$ of density $\alpha=\frac{1}{\sqrt{\log N}}$ such that $S(\mathcal{B}) \ll \vert \mathcal{B}\vert$ or in another words the multiplicative energy (as defined in \cite{Gowers4,Taogroup,TaoVu}) verifies
 
 $$ E_{\times}(\mathcal{B},N):=\vert \left\{ ab=cd, 1 \leq b,d \leq N, a,c \in \mathcal{B}\right\} \ll N \vert \mathcal{B}\vert.$$ This is of course optimal in terms of the size of the sets up to the constant. 
  We address the following problem
\begin{question} What is the maximal $0 < \alpha < 1$ (in terms of $N$) such that there exists a set $\mathcal{B}$ of density $\alpha$ verifying $ E_{\times}(\mathcal{B},N)\ll N \vert \mathcal{B}\vert$? \end{question}
 
\vspace{2mm} Our previous discussion shows that we can take $\alpha=\frac{1}{\sqrt{\log N}}$. We can ask the maybe easier question of constructing a set of density $\frac{1}{(\log N)^{\beta}}$ with $\beta  < 1/2$ having this property of minimizing the multiplicative energy. In order to get any improvement of our main result, the actual refined question would be sufficient
\begin{question} Can we construct a set $\mathcal{B}$ of density $\frac{1}{(\log N)^{\alpha}}$ such that $$ E_{\times}(\mathcal{B},N)\ll N \vert \mathcal{B}\vert (\log N)^{\beta} \hspace{4mm}\textrm{ with }\alpha + \beta < 1/2.$$   \end{question}
 Let us address the following related problem which might not have applications in the vanishing of theta functions but that we found interesting in its own right. Solymosi obtained in \cite{Solymosi} a beautiful upper bound $ E_{\times}(\mathcal{B},\mathcal{B}) \ll \vert\mathcal{B}+\mathcal{B}\vert^2 \log\vert \mathcal{B}\vert$ on the multiplicative energy of a set of reals. It implies particularly that it is good to choose a set with very small sumset in order to minimize the energy. We want to address the question whether we can get rid of this logarithmic term for sufficiently high-density sets of integers.
 \begin{question}What is the maximal $0 < \alpha < 1$ (in terms of $N$) such that there exists a set $\mathcal{B}$ of density $\alpha$ such that $ E_{\times}(\mathcal{B},\mathcal{B})\ll \vert \mathcal{B}\vert^2$? \end{question}

\section*{Acknowledgements}
The author would like to thank St\'{e}phane Louboutin and Igor Shparlinski for their comments on a first version of the draft as well as the latter one for pointing him out that similar ideas where used in \cite{Burgessrefine}. The author is supported by the Austrian Science Fund (FWF) project Y-901.

\end{document}